\documentclass[oneside,french,english]{amsart}
\usepackage[T1]{fontenc}
\usepackage[latin9]{inputenc}
\usepackage[letterpaper]{geometry}
\usepackage{amsthm}
\usepackage{setspace}
\usepackage{amssymb}

\makeatletter
\numberwithin{equation}{section}
\numberwithin{figure}{section}
  \theoremstyle{plain}
  \newtheorem*{thm*}{Theorem}
\theoremstyle{plain}
\newtheorem{thm}{Theorem}
  \theoremstyle{plain}
  \newtheorem{lem}[thm]{Lemma}
  \theoremstyle{plain}
  \newtheorem{prop}[thm]{Proposition}
  \theoremstyle{remark}
  \newtheorem*{rem*}{Remark}
\newcommand{\lyxrightaddress}[1]{
\par {\raggedleft \begin{tabular}{l}\ignorespaces
#1
\end{tabular}
\vspace{1.4em}
\par}
}

\makeatother

\usepackage{babel}
\addto\extrasfrench{\providecommand{\fg}{\ifdim\lastskip>\z@\unskip\fi~\frqq}}

\begin{document}
\selectlanguage{french}%
\textbf{\hfill{}}\foreignlanguage{english}{\textit{(English Transl.)
Paru à Journal of Mathematical Sciences (New York), 156:5, 819\textendash{}823
(2009).}}

\selectlanguage{english}%
\vspace{0.5cm}

\title{Toeplitz condition numbers as an $H^{\infty}$ interpolation problem}

\maketitle
\begin{center}
{\large Rachid Zarouf}
\par\end{center}{\large \par}
\begin{abstract}
The condition numbers $CN(T)=\left\Vert T\right\Vert .\left\Vert T^{-1}\right\Vert $
of Toeplitz and analytic $n\times n$ matrices $T$ are studied. It
is shown that the supremum of $CN(T)$ over all such matrices with
$\left\Vert T\right\Vert \leq1$ and the given minimum of eigenvalues
$r=min_{i=1..n}\vert\lambda_{i}\vert>0$ behaves as the corresponding
supremum over all $n\times n$ matrices (i.e., as $\frac{1}{r{}^{n}}$
(Kronecker)), and this equivalence is uniform in $n$ and $r$. The
proof is based on a use of the Sarason-Sz.Nagy-Foias commutant lifting
theorem. 
\end{abstract}
{\large Let $H$ be a Hilbert space of finite dimension $n$ and an
invertible operator $T$ acting on $H$ such that $\parallel T\parallel\leq1$.
We are interested in estimating the norm of the inverse of $T:$\[
\parallel T^{-1}\parallel.\]
More precisely, given a family $\mathcal{F}$ of $n-$dimensional
operators and a $T\in\mathcal{F}$, we set \[
r_{min}(T)=min_{i=1..n}\vert\lambda_{i}\vert>0,\]
where $\left\{ \lambda_{1},...,\lambda_{n}\right\} =\sigma(T)$ is
the spectrum of $T$. We are looking for {}``the best possible majorant''
$\Phi_{n}(r)$ such that}{\large \par}

\begin{onehalfspace}
{\large \[
\left\Vert T^{-1}\right\Vert \leq\Phi_{n}(r)\]
for every $T\in\mathcal{F}$, $\left\Vert T\right\Vert \leq1$. Let
$r\in]0,1[.$ This leads us to define the following bound $c_{n}(\mathcal{F},r),$ }{\large \par}
\end{onehalfspace}

\begin{doublespace}
{\large \[
c_{n}(\mathcal{F},r)=sup\left\{ \left\Vert T^{-1}\right\Vert :\: T\in\mathcal{F},\,\left\Vert T\right\Vert \leq1,\: r_{min}(T)\geq r\right\} \]
The following classical result is attributed to Kronecker ( XIX c.)}{\large \par}
\end{doublespace}
\begin{thm*}
\textbf{\large (}\textbf{\textup{\large Kronecker}}\textbf{\large ):}{\large \par}

{\large Let $\mathcal{F}$ be the set of all $n$-dimensional operators
defined on an euclidean space. Then \[
c_{n}(r):=c_{n}(\mathcal{F},r)=\frac{1}{r^{n}}\]
}{\large \par}
\end{thm*}
{\large Since obviously the upper bound in $c_{n}(r)$ is attained
(by a compactness argument), a natural question arises: how to describe
the extremal matrices $T$ such that $\left\Vert T\right\Vert \leq1$,
$r_{min}(T)\geq r$ and $\left\Vert T^{-1}\right\Vert =\frac{1}{r^{n}}$.
The answer is contained in N. Nikolski $[1]$ in the following form:
the case of equality \[
\left\Vert T^{-1}\right\Vert =\frac{1}{r^{n}}\]
 occurs for a matrix $T$ with $\parallel T\parallel=1$ if and only
if:}{\large \par}

{\large (1) either $r=1$ and then $T$ is an arbitrary unitary matrix.}{\large \par}

\begin{onehalfspace}
{\large (2) or $r<1$, and then the eigenvalues $\lambda_{j}(T)$
of $T$ are such that \[
\left|\lambda_{j}(T)\right|=r\]
and given $\sigma=\left\{ \lambda_{1},...,\lambda_{n}\right\} $ on
the circle, there exists a unique extremal matrix $T$ (up to a unitary
equivalence) with the spectrum $\left\{ \lambda_{1},...,\lambda_{n}\right\} $
having the form \[
T=U+K\]
where $K$ is a rank one matrix, $U$ is unitary and $U$ and $K$
are both given explicitly. (In fact, $T$ is nothing but the so-called
model operator corresponding to the Blaschke product $B=\Pi_{j=1}^{n}b_{\lambda_{j}}$
, see $[2]$ for definitions).}{\large \par}
\end{onehalfspace}

{\large For numerical analysis, the interest is in some classes of
structured matrices such as Toeplitz, Hankel etc.... In that note,
we are going to focus on the Toeplitz structure. Recall that $T$
is a Toeplitz matrix if and only if there exists a sequence $\left(a_{k}\right)_{k=-n+1}^{k=n-1}$
such that}{\large \par}

{\large \[
T=T_{a}=\left(\begin{array}{ccccc}
a_{0} & a_{-1} & . & . & a_{-n+1}\\
a_{1} & . & . & . & .\\
. & . & . & . & .\\
. & . & . & . & a_{-1}\\
a_{n-1} & . & . & a_{1} & a_{0}\end{array}\right),\]
and that $T$ is an analytic Toeplitz matrix if and only if there
exists a sequence $\left(a_{k}\right)_{k=0}^{k=n-1}$ such that}{\large \par}

{\large \[
T=T_{a}=\left(\begin{array}{ccccc}
a_{0} & 0 & . & . & 0\\
a_{1} & . & . & . & .\\
. & . & . & . & .\\
. & . & . & . & 0\\
a_{n-1} & . & . & a_{1} & a_{0}\end{array}\right).\]
We will denote by $\mathcal{T}_{n}$ the set of Toeplitz matrices
of size $n$ and $\mathcal{T}_{n}^{a}$ will be the set of analytic
Toeplitz matrices of size $n$. This leads us to the following questions.}{\large \par}

{\large How behave the constants $c_{n}(\mathcal{T}_{n},r)$ and $c_{n}(\mathcal{T}_{n}^{a},r)$
when $n\rightarrow\infty$ and/or $r\rightarrow0$? Are they uniformly
comparable with the Kronecker bound $c_{n}(r)$? Are there exist Toeplitz
matrices among extremal matrices described above? The answers seem
not to be obvious, at least the obvious candidates like $T=\frac{\lambda+J_{n}}{\left\Vert \lambda+J_{n}\right\Vert }$,
where $J_{n}$ is the $n-$dimensional Jordan matrix, do not lead
to the needed uniform (in $n$ and $r$) equivalence. For short, we
denote, }{\large \par}

{\large \[
t_{n}(r)=c_{n}(\mathcal{T}_{n},r)\]
and}{\large \par}

{\large \[
t_{n}^{a}(r)=c_{n}(\mathcal{T}_{n}^{a},r).\]
 Obviously we have,}{\large \par}

\begin{onehalfspace}
{\large \[
t_{n}^{a}(r)\leq t_{n}(r)\leq c_{n}(r)=\frac{1}{r^{n}}.\]
The following theorem answers the above questions.}{\large \par}
\end{onehalfspace}
\begin{thm*}
{\large 1) For all $r\in]0,1[$ and $n\geq1$, }{\large \par}

{\large \[
\frac{1}{2}\leq r^{n}t_{n}^{a}(r)\leq r^{n}c_{n}(r)=1\]
2) For every $n\geq1$}{\large \par}

{\large \[
lim_{r\rightarrow0}r^{n}t_{n}^{a}(r)=lim_{r\rightarrow1}r^{n}t_{n}^{a}(r)=1\]
and for every $0<r\leq1$ }{\large \par}

\begin{onehalfspace}
{\large \[
lim_{n\rightarrow\infty}r^{n}t_{n}^{a}(r)=1.\]
}\textup{\large The proof of the theorem is given in section 2 below.}{\large \par}\end{onehalfspace}

\end{thm*}

\section{\textbf{The operator $M_{n}$ and its commutant}}

{\large Let $M_{n}:\,\left(\mathbb{C}^{n},\,<.,.>\right)\longrightarrow\left(\mathbb{C}^{n},\,<.,.>\right)$
be the nilpotent Jordan Block of size $n$ \[
M_{n}=\left(\begin{array}{ccccc}
0\\
1 & .\\
 & . & .\\
 &  & . & .\\
 &  &  & 1 & 0\end{array}\right).\]
It is well known that the commutant $\left\{ M_{n}\right\} ^{'}=\left\{ A\in\mathcal{M}_{n}(\mathbb{C}):\, AM_{n}=M_{n}A\right\} $
of $M_{n}$ verifies \[
\left\{ M_{n}\right\} ^{'}=\left\{ p\left(M_{n}\right):\, p\in Pol_{+}\right\} ,\]
where $Pol_{+}$ is the space of analytic polynomials. On the other
hand, we can state this fact in the following way. Let \[
K_{z^{n}}=\left(z^{n}H^{2}\right)^{\perp}=Lin\left(1,\: z,...,\: z^{n-1}\right),\]
$H^{2}$ being the standard Hardy space in the disc $\mathbb{D}=\left\{ z:\,\vert z\vert<1\right\} $,
and \[
M_{z^{n}}:\, K_{z^{n}}\rightarrow K_{z^{n}}\]
 such that}{\large \par}

\begin{onehalfspace}
{\large \[
M_{z^{n}}f=P_{z^{n}}(zf),\,\forall f\in K_{z^{n}}.\]
 Then the matrix of $M_{z^{n}}$ in the orthonormal basis of $K_{z^{n}}$,
$B_{n}=\left\{ 1,\: z,...,\: z^{n-1}\right\} $ is exactly $M_{n}$,
and hence \[
\left\{ p\left(M_{n}\right),\, p\in Pol_{+}\right\} =\left\{ M_{n}\right\} ^{'}=\left\{ M_{z^{n}}\right\} ^{'}\]
The following straightforward link between $n\times n$ analytic Toeplitz
matrices and $\{M_{n}\}'$ is well known.}{\large \par}
\end{onehalfspace}
\begin{lem}
{\large $\mathcal{T}_{n}^{a}=\left\{ M_{n}\right\} ^{'}.$}{\large \par}\end{lem}
\begin{proof}
{\large Let \[
\phi(z)=\sum_{k\geq0}\hat{\phi}(k)z^{k}.\]
Then, \[
\phi\left(M_{n}\right)=\sum_{k=0}^{n-1}\hat{\phi}(k)M_{n}^{k}=\left(\begin{array}{ccccc}
\hat{\phi(}0)\\
\hat{\phi(}1) & .\\
. & . & .\\
. &  & . & .\\
\hat{\phi(}n-1) & . & . & \hat{\phi(}1) & \hat{\phi(}0)\end{array}\right).\]
Conversely, if $A=\left(\begin{array}{ccccc}
a_{0} & 0 & . & . & 0\\
a_{1} & . & . & . & .\\
. & . & . & . & .\\
. & . & . & . & 0\\
a_{n-1} & . & . & a_{1} & a_{0}\end{array}\right)\in\mathcal{T}_{n}^{a}$ then $A=\left(\sum_{k=0}^{n-1}a_{k}z^{k}\right)\left(M_{n}\right).$}{\large \par}
\end{proof}
{\large We also need the Schur-Caratheodory interpolation theorem
(1912), which also can be considered as a partial case of the commutant
lifting theorem of Sarason and Sz-Nagy-Foias (1968) see $[2]$ p.230
Theorem 3.1.11.}{\large \par}
\begin{prop}
{\large The following are equivalent.}{\large \par}

{\large i) $T$ is an $n\times n$ analytic Toeplitz matrix.}{\large \par}

{\large ii) There exists $g\in H^{\infty}$ such that $T=g\left(M_{n}\right)$.}{\large \par}

{\large Moreover}{\large \par}

{\large \[
\parallel T\parallel=inf\left\{ \parallel g\parallel_{\infty}:\, g\in H^{\infty}(\mathbb{D})\,,\, g\left(M_{n}\right)=T\right\} \]
}{\large \par}

{\large \[
=min\left\{ \parallel g\parallel_{\infty}:\, g\in H^{\infty}(\mathbb{D})\,,\, g\left(M_{n}\right)=T\right\} ,\]
where $\parallel g\parallel_{\infty}=sup_{z\in\mathbb{T}}\vert g(z)\vert.$ }{\large \par}
\end{prop}

\section{\textbf{Proof of the theorem}}
\begin{lem}
{\large Let $T$ be an invertible analytic Toeplitz matrix of size
$n\times n$ (which means that there exists $f\in Pol_{+}\subset H^{\infty}$
such that $T=f\left(M_{n}\right)$). Then}{\large \par}

{\large \[
\left\Vert T^{-1}\right\Vert =inf\left\{ \parallel g\parallel_{\infty}:\, g,\, h\in H^{\infty},\, fg+z^{n}h=1\right\} .\]
}{\large \par}\end{lem}
\begin{proof}
{\large Since $T^{-1}$belongs also to $\{M_{n}\}'$, there exists
$g\in Pol_{+}\subset H^{\infty}$ such that $T^{-1}=g\left(M_{n}\right)$.
This implies in particular that}{\large \par}

{\large \[
\left(fg\right)\left(M_{n}\right)=I_{n},\]
 which means that $fg-1$ annihilates $M_{n}$. That means that \[
fg-1\]
 is a multiple of $z^{n}$ in $H^{\infty}$. Conversely, if $g\in H^{\infty}$
verifies the above Bezout equation with $h\in H^{\infty}$ then}{\large \par}

{\large \[
g\left(M_{n}\right)=T^{-1}.\]
But by Proposition 1.2 we have}{\large \par}

\textit{\large \[
\left\Vert T^{-1}\right\Vert =inf\left\{ \parallel g\parallel_{\infty}:\, g\in H^{\infty},\, g\left(M_{n}\right)=T^{-1}\right\} ,\]
}{\large and hence}{\large \par}

\textit{\large \[
\left\Vert T^{-1}\right\Vert =inf\left\{ \parallel g\parallel_{\infty}:\, g,\, h\in H^{\infty},\, fg+z^{n}h=1\right\} .\]
}{\large \par}
\end{proof}
\textbf{\large Proof of the theorem.}{\large{} First, we prove that
for every $r\in]0,1[$ there exists an analytic $n\times n$ Toeplitz
matrix $T_{r}$ such that }{\large \par}

{\large 1 ) $\left\Vert T_{r}\right\Vert \leq1,$ }{\large \par}

{\large 2) $\sigma\left(T_{r}\right)=\left\{ r\right\} ,$}{\large \par}

{\large 3) $\left\Vert T_{r}^{-1}\right\Vert \geq\frac{1}{r^{n}}-1.$}{\large \par}

\begin{onehalfspace}
{\large Indeed, let \[
b_{r}(z)=\frac{r-z}{1-rz}\in H^{\infty}\]
 be the Blaschke factor corresponding to $r$. The $H^{\infty}$ calculus
of $M_{n}$ tells us that the operator \[
T_{r}:=b_{r}\left(M_{n}\right)\]
 satisfies property 1):}{\large \par}
\end{onehalfspace}

{\large \[
\left\Vert T_{r}\right\Vert \leq\left\Vert b_{r}\right\Vert {}_{\infty}=1.\]
 On the other hand, by the spectral mapping theorem }{\large \par}

\begin{onehalfspace}
{\large \[
\sigma\left(T_{r}\right)=\left\{ b_{r}\left(\sigma\left(M_{n}\right)\right)\right\} =\left\{ b_{r}(0)\right\} =\left\{ r\right\} .\]
In particular this proves that $T_{r}$ is invertible. Finally, using
Lemma 2.3, we get}{\large \par}
\end{onehalfspace}

\textit{\large \[
\left\Vert T_{r}^{-1}\right\Vert =inf\left\{ \parallel g\parallel_{\infty}:\, g,\, h\in H^{\infty},\, b_{r}g+z^{n}h=1\right\} =\]
}{\large \par}

{\large \[
=inf\left\{ \left\Vert \frac{1-z^{n}h}{b_{r}}\right\Vert _{\infty}:\, h\in H^{\infty},\, r^{n}h(r)=1\right\} =\]
}{\large \par}

\begin{onehalfspace}
{\large \[
=inf\left\{ \parallel1-z^{n}h\parallel_{\infty}:\, h\in H^{\infty},\, r^{n}h(r)=1\right\} .\]
But if $h\in H^{\infty}$and $r^{n}h(r)=1$, we have }{\large \par}
\end{onehalfspace}

{\large \[
\parallel1-z^{n}h\parallel_{\infty}\geq\parallel h\parallel_{\infty}-1\]
and}{\large \par}

{\large \[
\parallel h\parallel_{\infty}\geq\vert h(r)\vert=\frac{1}{r^{n}},\]
which gives}{\large \par}

{\large \[
\parallel1-z^{n}h\parallel_{\infty}\geq\frac{1}{r^{n}}-1.\]
Therefore}{\large \par}

\begin{onehalfspace}
{\large \[
\left\Vert T_{r}^{-1}\right\Vert \geq\frac{1}{r^{n}}-1,\]
which completes the proof of property 3) of $T_{r}$.}{\large \par}
\end{onehalfspace}

{\large Now we obtain}{\large \par}

\begin{onehalfspace}
{\large \[
1-r^{n}\leq r^{n}\left\Vert T_{r}^{-1}\right\Vert \parallel\leq r^{n}t_{n}^{a}(r)\leq r^{n}t_{n}(r)\leq r^{n}c_{n}(r)=1\]
for every $r\in]0,1[$. On the other hand, we have $\left\Vert T_{r}^{-1}\right\Vert \left\Vert T_{r}\right\Vert \geq1$
and hence}{\large \par}
\end{onehalfspace}

{\large \[
\left\Vert T_{r}^{-1}\right\Vert \geq\frac{1}{\left\Vert T_{r}\right\Vert }\geq1.\]
As a result for all $r\in]0,1[$, }{\large \par}

\begin{onehalfspace}
{\large \[
r^{n}\left\Vert T_{r}^{-1}\right\Vert \geq r^{n},\]
and combining with the previous estimate, we obtain}{\large \par}

{\large \[
\frac{1}{2}\le max(r^{n},1-r^{n})\leq r^{n}\left\Vert T_{r}^{-1}\right\Vert \leq r^{n}t_{n}^{a}(r)\leq r^{n}t_{n}(r)\leq r^{n}c_{n}(r)=1,\]
which completes the first claim of the theorem. The second claim follows
from the previous inequalities.}{\large \par}
\end{onehalfspace}

\begin{flushright}
{\large $\square$}
\par\end{flushright}{\large \par}
\begin{rem*}
{\large It should be mention that we have not obtained an explicit
formula for $t_{n}^{a}(r)$. Regarding the description of extremal
matrices (for the quantity $c_{n}(r)$) mentioned in the Introduction,
it seems likely that $t_{n}^{a}(r)<c_{n}(r)=\frac{1}{r^{n}}.$ In
the same spirit, it would be of interest to know the limits $lim_{r\rightarrow1}\left(inf_{n\geq1}r^{n}t_{n}^{a}(r)\right)$
and $lim_{n\rightarrow\infty}\left(inf_{0<r<1}r^{n}t_{n}^{a}(r)\right).$}{\large \par}\end{rem*}

\begin{onehalfspace}

\lyxrightaddress{{\large Equipe d'Analyse \& Géométrie, }}
\end{onehalfspace}

\lyxrightaddress{{\large Institut de Mathématiques de Bordeaux, }}

\lyxrightaddress{{\large Université Bordeaux, 351 Cours de la Libération, 33405 Talence,
France.}}

\lyxrightaddress{\textit{\large E-mail address}{\large : rzarouf@math.u-bordeaux1.fr}}
\end{document}